\documentclass[a4paper,12pt]{article}
\usepackage[centertags]{amsmath}
\usepackage{amsfonts}
\usepackage{amssymb}
\usepackage{amsthm}
\usepackage{amsmath}
\usepackage{dsfont}
\usepackage{graphicx}
\usepackage{tikz, subfigure}
\usepackage{pstricks-add}
\usepackage{verbatim}

\usepackage[latin1]{inputenc}
\usepackage{tikz}
\definecolor {processblue}{cmyk}{0.96,0,0,0}
\usepackage{amsfonts,graphicx,amsmath,amssymb,hyperref,color}
\usepackage[english]{babel}
\usepackage{authblk}

\newtheorem{Theorem}{Theorem}[section]

\newtheorem{Lemma}{Lemma}[section]

\AtEndDocument{\bigskip{\footnotesize%
  \textsc{Department of Mathematics, Ningbo University, Ningbo, Zhejiang, People's Republic of China} \par
  \textit{E-mail address:} \texttt{283186364@qq.com} \par
}}
\AtEndDocument{\bigskip{\footnotesize%
  \textsc{Department of Mathematics, Ningbo University, Ningbo, Zhejiang, People's Republic of China} \par
  \textit{E-mail address:} \texttt{740464942@qq.com} \par
}}
\AtEndDocument{\bigskip{\footnotesize%
  \textsc{Department of Mathematics, Ningbo University, Ningbo, Zhejiang, People's Republic of China} \par
  \textit{E-mail address:} \texttt{2017084162@qq.com} \par
}}
\AtEndDocument{\bigskip{\footnotesize%
  \textsc{Department of Mathematics, Ningbo University, Ningbo, Zhejiang, People's Republic of China} \par
  \textit{E-mail address:} \texttt{kanjiangbunnik@yahoo.com, jiangkan@nbu.edu.cn} \par
}}
\usepackage{lipsum}

\makeatletter
\newcommand*{\rom}[1]{\expandafter\@slowromancap\romannumeral #1@}
\makeatother

\begin{document}
\title{xxxx}
\date{}
 \title{Arithmetic on Moran sets}
\author{Xiaomin Ren, Li Tian, Jiali Zhu and Kan Jiang\thanks{Kan Jiang is the corresponding author}}
\maketitle{}
\begin{abstract}
Let $(\mathcal{M}, c_k,n_k)$ be a class of Moran sets.   We assume that the convex hull of any $E\in (\mathcal{M}, c_k,n_k)$ is  $[0,1]$.
Let $A,B$ be two non-empty sets in $\mathbb{R}$.  
Suppose that $f$
is a continuous function defined on an open set $U\subset \mathbb{R}^{2}$.
Denote the  continuous image  of $f$  by 
\begin{equation*}
f_{U}(A,B)=\{f(x,y):(x,y)\in (A\times B)\cap U\}.
\end{equation*}
In this paper, we prove the following result.
Let $E_1,E_2\in(\mathcal{M}, c_k, n_k)$. If there exists some $(x_0,y_0)\in (E_1\times E_2)\cap U$ such that 
$$\sup_{k\geq 1}\left\{1-c_kn_k\right\}<\left\vert \frac{\partial
_{y}f|_{(x_{0},y_{0})}}{\partial _{x}f|_{(x_{0},y_{0})}}\right\vert <\inf_{k\geq 1}\left\{\dfrac{c_k}{1-n_kc_k}\right\}.$$
Then $f_U(E_1, E_2)$ contains an interior. 
\end{abstract}
\section{Introduction}

Given two non-empty sets $A,B\subset\mathbb{R}$. Define $A*B=\{x*y:x\in A, y\in B\}$, where $*$ is  $+, -,\times$ or $\div$ (when $*=\div$, $y\neq 0$). We call $A*B$ the arithmetic on $A$ and $B$.  Generally, we may define  the arithmetic on $A$ and $B$ in terms of some functions.
Suppose that $f$
is a continuous function defined on an open set $U\subset \mathbb{R}^{2}$.
Denote the  continuous image  of $f$  by 
\begin{equation*}
f_{U}(A,B)=\{f(x,y):(x,y)\in (A\times B)\cap U\}.
\end{equation*}
For simplicity, we still call $f_{U}(A,B)$ the  arithmetic on $A$ and $B$. 
Arithmetic on the fractal sets has  strong connections with many different problems in geometry measure theory and dynamical systems \cite{XiH, WX}. For instance, in geometry measure theory, the visible problem is related to  the division on the fractals \cite{FF,JJ1,Neil}. The main reason is due to the following observation.   Let $K\subset [0,1]$ be a fractal set. 
 Given $\alpha \geq  0$, we say the line $y=\alpha x$ is visible through $K\times K$ if
$$
\{(x, \alpha x): x\in \mathbb R\setminus \{0\}\}\cap (K\times K)=\emptyset.
$$
It is easy to verify  that the line $y=\alpha x$ is visible
through $K\times K$ if  and only if
$$\alpha\notin \frac{K}{
K}:=\left\{\dfrac{x}{y}:x,y\in K, y\neq 0\right\}.$$
The arithmetic sum  of two Cantor sets was studied by many scholars. There are many results concerning with this topic, see \cite {Yoccoz,Dekking1, Eroglu,  Hochman2012, SumKan, Hall,PS}  and references therein.  It is an important problem in  homoclinic bifurcations \cite{Palis}. Palis \cite{Palis} posed the following problem: whether it is true (at least generically) that the arithmetic  sum of dynamically defined Cantor sets either has measure zero or contains an interval. This conjecture was solved in \cite{Yoccoz}.  Motivated by Palis' conjecture, it is natural to investigate when the sum of two Cantor sets contains some interiors. Newhouse \cite{SN} proved the following thickness theorem. Given any two Cantor sets $C_1$ and $C_2$,   if $\tau(C_1)\tau(C_2)>1$, where $\tau(C_i),i=1,2$ denotes the thickness of $C_i, i=1,2,$ then $C_1+C_2$ contains some interiors.  
However, Newhouse thickness theorem  cannot handle  a general function $f$, i.e. whether $f(C_1,C_2)$ contains an interior or not. 

To date, there are not so many results concerning with the arithmetic on the fractal sets \cite{Tyson,HS,Yuki}.  The first result of this direction, to the best of our knowledge, is due to Steinhaus \cite{HS} who proved the following interesting result:
$C-C=[-1,1]$, where $C$ is the middle-third Cantor set.  Equivalently, Steinhaus proved that for any $x\in[-1,1]$, there are some $x_1,x_2\in C$ such that $x=x_1-x_2.$ Recently, Athreya,  Reznick and Tyson  \cite{Tyson} considered the multiplication on the middle-third Cantor set.   They proved that $17/21\leq \mathcal{L}(C\cdot C)\leq 8/9$, where $\mathcal{L}$ denotes the Lebesgue measure. Jiang and Xi  \cite{XiKan3} proved that  $C\cdot C$ indeed contains infinitely  many intervals. 
 In \cite{XiKan2}, Jiang and Xi  considered the representations of  real numbers  in $ C-C=[-1,1]$, i.e. let $x\in [-1,1]$, define 
\begin{equation*}
S_{x}=\left\{ \mathbf{(}y_{1},y_{2}\mathbf{)}:y_{1}-y_{2}=x,\,\, (y_1,y_2)\in C\times C\right\} .
\end{equation*}
and 
\begin{equation*}
U_{r}=\{x:\sharp(S_{x})=r\},  r\in \mathbb{N}^{+}. 
\end{equation*}
They proved that $\dim_{H}(U_r)=\dfrac{\log 2}{\log 3}$ if  $r=2^k$ for some $k\in \mathbb{N}$. Moreover, $$0<\mathcal{H}^{s}(U_1)<\infty, \mathcal{H}^{s}(U_{ 2^k})=\infty,k\in \mathbb{N}^{+},$$ 
where $s=\dfrac{\log 2}{\log 3}$. 
$U_{3\cdot 2^k}$ is an infinitely countable set for any $k\geq1$, where $\dim_{H}$ and $\mathcal{H}^{s}$ denote the Hausdorff dimension and Hausdorff measure, respectively. For more results, see \cite{XiKan2}.  In \cite{XiKan1}, Tian et al. defined a class of overlapping self-similar sets as
follows: let $K$ be the attractor of the IFS
\begin{equation*}
\{f_{1}(x)=\lambda x,f_{2}(x)=\lambda x+c-\lambda ,f_{3}(x)=\lambda
x+1-\lambda \},
\end{equation*}
where $f_{1}(I)\cap f_{2}(I)\neq \emptyset ,(f_{1}(I)\cup f_{2}(I))\cap
f_{3}(I)=\emptyset ,$ and $I=[0,1]$ is the convex hull of $K$. This class of self-similar set was investigated by many scholars, see \cite{FengLau, Hochman, XiA,  XiB, LTT11, XiG,XiH}. 
Tian et al.  $K\cdot
K=[0,1]$ if and only if $(1-\lambda )^{2}\leq c$. Equivalently, they gave a
necessary and sufficient condition such that for any $x\in \lbrack 0,1]$
there exist some $y,z\in K$ such that $x=yz$.
Moreover, Ren, Zhu, Tian and Jiang \cite{RZTJ} proved that 
$$\sqrt{K}+\sqrt{K}=[0,2]$$ if and only if 
$$\sqrt{c}+1\geq 2\sqrt{1-\lambda},$$
where $\sqrt{K}+\sqrt{K}=\{\sqrt{x}+\sqrt{y}:x,y\in K\}$.
If $c\geq (1-\lambda)^2$, then $$\dfrac{K}{K}=\left\{\dfrac{x}{y}:x,y\in K, y\neq 0\right\}=\left[0,\infty\right).$$ As a consequence, they  proved that the  following conditions are equivalent:
\begin{itemize}
\item [(1)] For any $u\in [0,1]$, there are some $x,y\in K$ such that $u=x\cdot y;$
\item [(2)] For any $u\in [0,1]$, there are some $x_1,x_2,x_3,x_4,x_5,x_6,x_7,x_8,x_9,x_{10}\in K$ such that $$u=x_1+x_2=x_3-x_4=x_5\cdot x_6=x_7\div x_8=\sqrt{x_9}+\sqrt{x_{10}};$$
\item [(3)] $c\geq (1-\lambda)^2$.
\end{itemize}
In this paper, we shall consider similar problems on  the Moran sets. The Moran sets are, in certain sense, random. Nevertheless, any self-similar set with the open set condition is a Moran set \cite{Hutchinson}. Now we give the definition of  a class of Moran set. 
Let $\{n_k\}\subset \mathbb{N}^{+}$ be a sequence(we assmue that $n_k\geq 2$). For any $k\in \mathbb{N}^{+}$, write $$D_k=\{(\sigma_1,\cdots, \sigma_k): \sigma_j\in \mathbb{N}^{+}, 1\leq \sigma_j\leq n_j,1\leq j\leq k\}.$$
Define $$D=\cup_{k\geq 0}D_k.$$
We call $\sigma \in D$ a word. For simplicity, we let $D_{0}=\emptyset.$ If $\sigma=(\sigma_1,\cdots, \sigma_k)\in D_k$, $\tau=(\tau_1,\cdots, \tau_m)\in D_m$,  then we  define the concatenation  $\sigma*\tau=(\sigma_1,\cdots, \sigma_k,\tau_1,\cdots, \tau_m)\in D_{k+m}$. 
Let $T=[0,1]$ and   $\{c_k\}$ be a positive real sequence with $c_kn_k<1, k\in \mathbb{N}^{+}$, we say the class $$\mathcal{F}=\{T_{\sigma}\subset T:\sigma\in D\}$$ has the Moran structure if the following conditions are satisfied:
\begin{itemize}
\item [(1)] for any $\sigma\in D$, $T_{\sigma}$ is similar to $T$, i.e. there exists a similitude $S_{\sigma}: \mathbb{R}\to \mathbb{R}$ such that $S_{\sigma}(T)=T_{\sigma};$ 
\item [(2)] for any $k\geq 0$ and $\sigma\in D_k$, $T_{\sigma*1}, T_{\sigma*2},\cdots,T_{\sigma*n_{k+1}}$ is a subset of $T_{\sigma}$ and $$int (T_{\sigma*i})\cap int (T_{\sigma*j})=\emptyset, i\neq j,$$ where $int (A)$ denotes the interior of $A$, for simplicity, we denote by $\widetilde{T_{\sigma}}=\cup_{i=1}^{n_{k+1}}T_{\sigma*i}$;
\item [(3)] for any $k\geq 1$ and $\sigma\in D_{k-1}$, 
$\dfrac{|T_{\sigma*i}|}{|T_{\sigma}|}=c_k$, and the convex hull of $T_{\sigma*i}$ and $T_{\sigma}$ coincide for any $1\leq i\leq n_k$,  where $|A|$ denotes the diameter of $A$. 
\end{itemize}
Suppose  $\mathcal{F}=\{T_{\sigma}\subset T:\sigma\in D\}$  has the Moran structure, then we call 
$$E=\cap_{k\geq 1}^{\infty}\cup_{\sigma\in D_k} T_{\sigma}$$ a Moran set. We denote by $(\mathcal{M}, c_k,n_k)$ all the Moran sets generated by the Moran structure $\mathcal{F}.$
By the third  condition, it is easy to see that the convex hull of any $E$ from $(\mathcal{M}, c_k,n_k)$ is $[0,1].$

Now we are ready to state the main result of this paper. 
\begin{Theorem}\label{Main}
Let $E_1,E_2\in(\mathcal{M}, c_k, n_k)$. If there exists some $(x_0,y_0)\in (E_1\times E_2)\cap U$ such that 
$$\sup_{k}\left\{1-c_kn_k\right\}<\left\vert \frac{\partial
_{y}f|_{(x_{0},y_{0})}}{\partial _{x}f|_{(x_{0},y_{0})}}\right\vert <\inf_{k}\left\{\dfrac{c_k}{1-n_kc_k}\right\}.$$
Then $f_U(E_1, E_2)$ contains an interior. 
\end{Theorem}
The paper is arranged as follows. In section 2, we prove two basic lemmas and give a proof of Theorem \ref{Main}. In section 3, we give some remarks. 
\section{Proof of Theorem \ref{Main}}
In this section, we shall prove Theorem \ref{Main}.  First, we  give some definitions and prove two useful lemmas. 
 
 For any $k\geq 1$, denote  by  $E_k$ the union of basic intervals when we construct a Moran set $E$, i.e. 
 $$E_k=\cup_{\sigma\in D_k}T_{\sigma}, E=\cap_{k=1}^{\infty}E_k,$$ where $T_{\sigma}$ is called a basic interval with rank $k$. 
 It is easy to check that the length of any basic interval with rank $k$ is $c_1c_2\cdots c_k.$
 Let $[A,B]\subset [0,1]$, where $A$ and $B$ are  the  left and right endpoints of  some basic intervals in $E_k$ for some $k\geq 1$, respectively. $A$ and $B$ may not in the same basic interval.  In the following lemma, we choose $A$ and $B$ in this way.  Let $F_k$ be the collection of all the basic intervals in $[A,B]$ with length  $c_1c_2\cdots c_k, k\geq k_0$ for some $k_0\in\mathbb{N}^{+}$, i.e.  the union of all the elements of $F_k$ is denoted by $G_k=\cup_{i=1}^{t_k}I_{k,i}$, where $t_k\in \mathbb{N}^{+}$,  $I_{k,i}\subset  E_k\cap  [A,B]$. Clearly, by the definition of $G_n$, it follows that $G_{n+1}\subset G_n$ for any $n\geq k_0.$
\begin{Lemma}\label{key1}
Let $E_1, E_2\in (\mathcal{M}, c_k, n_k)$, i.e. 
$$E_1=\cap_{k=1}^{\infty}E_k^{(1)}, E_2=\cap_{k=1}^{\infty}E_k^{(2)}.$$
Assume $F:\mathbb{R}^2\to \mathbb{R}$ is a continuous function.  
Suppose $A$ and $B$ ($C$ and $D$) are  the  left and right endpoints of  some basic intervals in $E_{k_0}^{(1)}$($E_{k_0}^{(2)}$) for some $k_0\geq 1$, respectively. 
Then  $E_1\cap [A,B]=\cap_{n={k_0}}^{\infty}G_n^{(1)}, E_2\cap [C,D]=\cap_{n={k_0}}^{\infty}G_n^{(2)}$.
Moreover, if   for any $n\geq k_0$ and any basic intervals $I_1\subset G_n^{(1)}, I_2\subset G_n^{(2)}$,  we have 
$$F(I_1, I_2)=F(\widetilde{I_1}, \widetilde{I_2}),$$
then $F(E_1\cap [A,B],E_2\cap [C,D] )=F(G_{k_0}^{(1)}, G_{k_0}^{(2)}).$ 
\end{Lemma}
\begin{proof}
We assume that 
$G_n^{(i)}=\cup_{1\leq j\leq t_n^{(i)}}I_{n,j}, i=1,2.$ 
By the construction of $G_n^{(i)}, i=1,2$,  it is clear that $G_{n+1}^{(i)}\subset G_n^{(i)}$ for any $n\geq 1. $
Therefore, 
$$E_1\cap [A,B]=\cap_{n={k_0}}^{\infty}G_n^{(1)}, E_2\cap [C,D]=\cap_{n={k_0}}^{\infty}G_n^{(2)}.$$
In terms of the  continuity of $F$, we conclude that 
\begin{eqnarray}\label{identity}
F(E_1\cap [A,B],E_2\cap [C,D] )=\cap_{n=k_0}^{\infty}F(G_{n}^{(1)}, G_{n}^{(2)}).
\end{eqnarray}
Therefore, 
\begin{eqnarray*}
F(G_n^{(1)}, G_n^{(2)})&=& \cup_{1\leq i\leq t_n^{(1)}, 1\leq j\leq t_n^{(2)}}F(I_{n,i}, I_{n,j})\\&=
&\cup_{1\leq i\leq t_n^{(1)}, 1\leq j\leq t_n^{(2)}}F(\widetilde{I_{n,i}},\widetilde{I_{n,j}})\\&=&
F(\cup_{1\leq i\leq t_n^{(1)}}\widetilde{I_{n,i}},\cup_{1\leq j\leq t_n^{(2)}}\widetilde{I_{n,j}})\\&=&F(G_{n+1}^{(1)}, G_{n+1}^{(2)}).
\end{eqnarray*}
Consequently,  $F(E_1\cap [A,B],E_2\cap [C,D] )=F(G_{k_0}^{(1)}, G_{k_0}^{(2)})$  follows immediately   from   the identity (\ref{identity}) and $F(G_n^{(1)}, G_n^{(2)})=F(G_{n+1}^{(1)}, G_{n+1}^{(2)})$ for any $n\geq k_0.$ 
\end{proof}
\begin{figure}
  \centering
    \includegraphics[width=450pt]{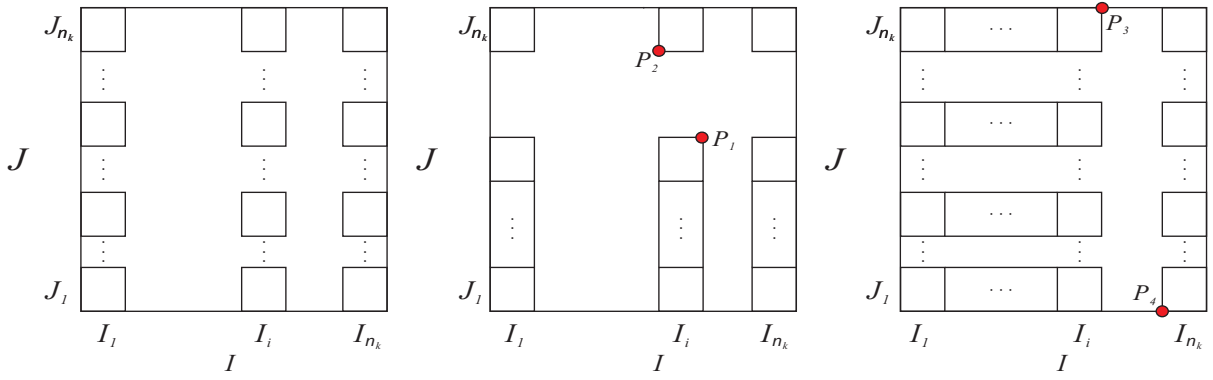}
   \caption{  The figures of $\widetilde{I}\times \widetilde{J}$ }
\end{figure}

\begin{Lemma}\label{key2}
 Let $I=[a,a+t], J=[b,b+t]$ be two basic intervals in  $G_{k-1}^{(1)}$ and $G_{k-1}^{(2)}$, respectively. 
 If there exists some $(x_0,y_0)\in (E_1\times E_2)\cap (I\times J)\cap U$ such that 
$$\sup_{k}\left\{1-c_kn_k\right\}<\left\vert \frac{\partial
_{y}f|_{(x_{0},y_{0})}}{\partial _{x}f|_{(x_{0},y_{0})}}\right\vert <\inf_{k}\left\{\dfrac{c_k}{1-n_kc_k}\right\}.$$
Then $f(I,J)=f(\widetilde{I},\widetilde{J})$. 
\end{Lemma}
\begin{proof}
Without loss of generality, we assume that $\partial
_{x}f|_{(x_{0},y_{0})}>0, \partial
_{y}f|_{(x_{0},y_{0})}>0$. For other cases, we may consider the new function $F(x,y)=f(x,1-y)$ or $-f(x,y)$.
By the definition of $\widetilde{I}$ and $\widetilde{J}$, we have 
$$\widetilde{I}=\cup_{i=1}^{n_k}I_{i}, \widetilde{J}=\cup_{j=1}^{n_k}J_{j}.$$
Moreover, $t=|I|=|J|=c_{1}\cdots c_{k-1}$, where $|A|$ denotes the length of $A$. 
Therefore, we have 
$$f(\widetilde{I},\widetilde{J})=\cup_{i=1}^{n_k}\cup_{j=1}^{n_k}f(I_{i},J_{j}).$$
We first prove that  for any $1\leq i\leq n_k$, $\cup_{j=1}^{n_k}f(I_{i},J_{j})$ is an interval.
By the construction of Moran set,  it suffices to prove that  $f(P_1)\geq f(P_2)$, see the second picture of Figure 1, that is, it remains to prove that 
there exists some  $(\xi,\eta)\in E_1\times E_2$ contained in the 
neighbour of $(x_0,y_0)$  such that  
$$  (c_1\cdots c_k)\partial
_{x}f(\xi,\eta)\geq (c_1c_2\cdots c_{k-1}-n_kc_1\cdots c_k) \partial
_{y}f(\xi,\eta).$$
However, this is clear due to the condition$$  \frac{\partial
_{y}f|_{(x_{0},y_{0})}}{\partial _{x}f|_{(x_{0},y_{0})}} <\inf_{k}\left\{\dfrac{c_k}{1-n_kc_k}\right\},$$ and the assumption  $\partial _{x}f, \partial _{y}f$ are continuous. 
Next, we prove that $$\cup_{i=1}^{n_k}\cup_{j=1}^{n_k}f(I_{i},J_{j})$$ is an interval. 
Analogously, we need to show that $f(P_3)\geq f(P_4)$, see the third picture of Figure 1. 
Indeed, it only remains to prove that there is some  $(\xi_1,\eta_1)\in E_1\times E_2$ which lies in the 
neighbour of $(x_0,y_0)$  such that  
$$ (c_1\cdots c_{k-1} ) \partial
_{y}f(\xi_1,\eta_1)\geq  (c_1c_2\cdots c_{k-1}-n_kc_1\cdots c_k)\partial
_{x}f(\xi_1,\eta_1).$$
However, the above inequality follows from the condition
$$
\sup_{k}\left\{1-c_kn_k\right\}<\frac{\partial
_{y}f|_{(x_{0},y_{0})}}{\partial _{x}f|_{(x_{0},y_{0})}},$$
and $\partial _{x}f, \partial _{y}f$ are continuous. 
Therefore, we have proved that $f(I,J)=f(\widetilde{I},\widetilde{J})$. 
\end{proof}
\begin{proof}[\textbf{Proof of Theorem \ref{Main}}]
Theorem  \ref{Main} follows immediately from Lemmas   \ref{key1} and \ref{key2}. 
\end{proof}
\section{Final remark}
In Lemma \ref{key1}, we note that if some basic intervals of $E_k$ intersects, then similar result as Theorem \ref{Main} can be obtained. We leave it to the readers.

  \section*{Acknowledgements}
The work is supported by National Natural Science Foundation of China (Nos.11701302,

11671147). The work is also supported by K.C. Wong Magna Fund in Ningbo University.

\end{document}